\documentclass[pdflatex,sn-mathphys-num]{sn-jnl}% Math and Physical Sciences Numbered Reference Style
\usepackage{graphicx}%
\usepackage{multirow}%
\usepackage{amsmath,amssymb,amsfonts}%
\usepackage{amsthm}%
\usepackage{mathrsfs}%
\usepackage[title]{appendix}%
\usepackage{xcolor}%
\usepackage{textcomp}%
\usepackage{manyfoot}%
\usepackage{booktabs}%
\usepackage{algorithm}%
\usepackage{algorithmicx}%
\usepackage{algpseudocode}%
\usepackage{listings}%
\theoremstyle{thmstyleone}%
\newtheorem{theorem}{Theorem}%  meant for continuous numbers
\newtheorem{corollary}[theorem]{Corollary}% 

\theoremstyle{thmstyletwo}%

\theoremstyle{thmstylethree}%

\begin{document}

\title[A Markov Chain Arising from the Hopf Square Map on a Non-cocommutative Quantum Group]{A Markov Chain Arising from the Hopf Square Map on a Non-cocommutative Quantum Group}


\author{\fnm{Donovan} \sur{Snyder}}\email{dsnyd15@ur.rochester.edu}

\affil{\orgdiv{Math Department}, \orgname{University of Rochester}, \orgaddress{\street{500 Joseph C. Wilson Blvd.}, \city{Rochester}, \postcode{14627}, \state{NY}, \country{USA}}}


\abstract{Expanding upon the rich history of algebraic techniques in probability, we show the existence of and construct a Markov chain using the Hopf square map on a quantum group that is both non-commutative and non-cocommutative.
This extends the work of Diaconis, Pang, and Ram to other Hopf algebras.
The new, one-dimensional chain requires different analytical approaches. In this case we use standard martingale theory to prove the existence of a phase transition and prove bounds on the expected growth rates.}




\maketitle

\bmhead{Acknowledgements}

I would like to thank Sarada Rajeev, Jonathan Pakianathan, and Arjun Krishnan for their extensive help, discussions, and feedback throughout this work.

\section{\label{Introduction}Introduction}

Algebraic techniques have significantly expanded the number of models in probability theory and statistical mechanics that can be completely solved.
Even further, many models exhibit a phase transition which was originally found through the use of algebraic techniques.
The Ising model is a prime example, where the transfer matrix was used to completely solve the one-dimensional nearest-neighbor model \cite{Ising1925, KramersWannier}.
Algebraic structures, somewhat similar to those used in this paper, also played a key role in the construction and solving of the Potts model \cite{Potts1952}.
However, these models do not exhibit a phase transition in one dimension.
% : in general, it is rare to see phase transitions in one dimensional models (for example, see the Mermin-Wagner theorem \cite{MWT}).

On the other hand, the Ising model in two dimensions does exhibit a phase transition, as shown in Onsager's seminal result \cite{Onsager1944} (also see \cite{Kaufman1949}) using the irreducible representations of a matrix algebra.
Other algebraic and combinatorial techniques were used to study this model, such as the use of transfer matrices \cite{SchultzLieb1964} and (what came to be known as) the Yang-Baxter equation \cite{BaxterEnting1978}.
The Yang-Baxter equation has been used throughout probability theory more recently, particularly in the study of the stochastic six-vertex model \cite{SixVertex} and related integrable systems \cite{YangBaxterUse}.

Our goal is to examine and expand the use of Hopf algebras and the associated structures.
To do so, we build on the work of Diaconis, Pang, and Ram in \cite{DiaconisPangRam} who use Hopf algebras in a seemingly independent way from their use in the Yang-Baxter equation.
The Hopf square map is a natural linear operator on the Hopf algebra, and its matrix representation is sometimes the Markov matrix for standard processes like card shuffling and rock breaking.
This result follows from the examination of how the Hopf square behaves on familiar Hopf algebras.
Further, by using properties of the Hopf algebra, they can study the chains in new ways.

In this paper, we construct and analyze a Markov chain that arises from the Hopf square map on a non-commutative and non-cocommutative Hopf algebra.
This type of Hopf algebra is not analyzed in \cite{DiaconisPangRam} (see their Example 6.6) and to our knowledge is not addressed elsewhere in the literature.

We use the Hopf algebra \(U_q(\mathfrak{sl}_2)\), a quantum deformation of the universal enveloping algebra of the Lie algebra \(\mathfrak{sl}_2\).
That we can construct a Markov chain from the Hopf square is not obvious because the entries of such a matrix are not guaranteed to be non-negative.
The result is a one-dimensional growth process $X_n$ on $\{0,1,2,\ldots\}$ with transition probabilities depending on a deformation parameter $q$, given by
\begin{equation}\label{eq:probs}
    T(i,i+1) = \frac{1}{q^{i}+1} \qquad T(i,i) = \frac{q^{i}}{q^{i}+1}
\end{equation}
and initial condition $X_0=0$.

Unlike Diaconis et al., we are unable to use the properties of the Hopf algebra to study the resulting chain.
In their case, the structure of the Hopf algebras allowed for the diagonalization of the Markov chain.
Since we no longer have these assumptions, we must use more standard techniques.
Using martingale methods, we are able to characterize a phase transition as the deformation parameter $q$ is varied.
\begin{theorem}\label{thm:introTheorem}
    There exists a Markov chain \(X_n\) that comes from the Hopf square map on \(U_q(\mathfrak{sl}_2)\).
    The Markov chain $X_n$ exhibits a phase transition:
    \begin{equation}
        \lim_{n\rightarrow\infty}\frac{\mathbb{E}(X_{n})}{n}=
        \begin{cases}
            1           & 0<q<1 \\
            \frac{1}{2} & q=1   \\
            0           & q>1
        \end{cases}
    \end{equation}
\end{theorem}
We also prove direct bounds in the two different regimes:
\begin{theorem}
    For the Markov Chain $X_n$, there exists constants $C,D,N$ such that for $n\geq N$,
    \begin{align*}
        X_n & \leq C\ln (n), & q>1   \\
        X_n & \geq n - D,    & 0<q<1
    \end{align*}
\end{theorem}
We generalize the above results to a larger class of growth processes that have similar transition probabilities.
The existence of such a phase transition is shown in Theorem \ref{thm:phase}. We also prove direct bounds on those models in Theorem \ref{thm:gen}.

The structure of the paper is as follows: we briefly review the work of \cite{DiaconisPangRam}, detailing how a Markov chain can be extracted from the Hopf square.
We will largely use their Kassel's \cite{QG} notation in our discussion of Hopf algebras, the latter of which is a good source for background on these objects.
Section \ref{sec:NewWork} contains our novel contribution:
we introduce the non-commutative and non-cocommutative Hopf algebra \(U_q(\mathfrak{sl}_2)\) that is the foundation of our work and identify the Markov chain that arises from the Hopf square.
We then examine the probability distribution and describe the phase transition that occurs as the deformation parameter $q$ is varied.
We conclude with some future work in Section \ref{sec:future} before giving the proofs of our results in Section \ref{sec:Proofs}.

\section{The Hopf Square Forming a Markov Chain}

A Hopf algebra $H$ is a vector space with a compatible algebra and coalgebra structure.
Combining the key operations of the multiplication and comultiplication, \(H\) has the linear map formed by the composition $\Psi^2:=\mu\circ\Delta$.
This is called the \textit{Hopf square}.
% and can be thought of as a process that sends elements of the Hopf algebra to others.
Explicitly, from Sweedler's notation for the coproduct (see \cite{QG} for details),
% in equation \ref{eq:Sweedler},
$$\Psi^2(x)=\sum_{(x)} x_{(1)}x_{(2)}$$ is a finite sum of products depending on \(x\).
A graded Hopf algebra is a direct sum $H= \bigoplus_n H_n$ where both the multiplication and comultiplication respect the grading.
Thus the Hopf square map respect the grading: if $x\in H_n$ then $\Psi^2(x)\in H_n$, too.
This allows us to examine each grading $H_n$ separately.

On any grading $H_n$ of a graded Hopf algebra $H$, take a basis $\mathfrak {B}_n=\{b_1,b_2,\ldots\}$ (we will only work with countable bases).
Then, for any basis element $b$, $\Psi^2(b)$ can be written as a unique linear combination of basis elements.
For certain Hopf algebras and certain choice of basis, the coefficients in the linear combination can be used to create the transition matrix of a Markov chain on the space of the basis elements-- this is summarized in the theorem below.

First, sum the coefficients of $\Psi^2(b)$ to get a normalization $N_b$.
Then, dividing by this normalization leads to the probability $p(b,b_i)$ of transition between $b$ and $b_i$.
Equivalently, the linear combination can be written as $\Psi^2(b)=\sum_i N_b p(b,b_i)b_i$.
A normalization like $N_b$ is often seen in Markov chains arising from graphs on non-regular graphs.
Of course, this representation is heavily dependent on the choice of basis.

A key result in the work of \cite{DiaconisPangRam} shows that in specific, familiar Hopf algebras, a specific choice of basis leads to all of the coefficients being non-negative and each $N_b$ being equal.
\begin{theorem}[Diaconis, Pang, Ram]\label{thm:Diac}
    Let $H=\bigoplus H_n$ be a graded Hopf algebra over $\mathbb R$ which, as an algebra, is either a polynomial algebra or is a cocommutative free associative algebra.

    Then there exists a basis $\mathfrak {B}_n$ of $H_n$ such that the matrix representation of the normalized Hopf square $\frac{1}{2^n}\Psi^2$ in that basis, transposed, is a transition matrix.
    That is, $T_n$ is a transition matrix for a Markov chain on $\mathfrak {B}_n$ as defined by
    \begin{equation*}
        \frac{1}{2^n}\Psi^2(b) = \sum_{b'\in \mathfrak {B}_n} T_n(b,b')b'
    \end{equation*}
\end{theorem}

It turns out that applying this theorem to a few common Hopf algebras gives common Markov chains: the free associative algebra gives rise to a deck shuffling chain, and the symmetric polynomial Hopf algebra gives rise to the rock breaking chain.
More results in their work allow for the structure of the Hopf algebras to be used to analyze the chains.

Unfortunately, the assumptions of this theorem eliminate the Hopf algebras that we would like to study: quantum groups that are both non-commutative and non-cocommutative.

\section{\label{sec:NewWork}The Hopf Square on a Quantum Group}

\subsection{The Hopf algebra}

The non-commutative and non-cocommutative Hopf Algebra we study is a quantum deformation of $\mathfrak{sl}(2)$.
For $q\in \mathbb{R}$, define $U_{q}$ as a free algebra with generators $E,F,K,K^{-1}$, subject to the following multiplication relations (where \(xy:=\mu(x,y)\))
\begin{alignat*}{1}
    KK^{-1}=                       & K^{-1}K=1         \\
    KE=q^{2}EK \quad               & \quad KF=q^{-2}FK \\
    (q-q^{-1})\left(EF - FE\right) & = (K-K^{-1})
\end{alignat*}
% The last fraction is defined since $q\neq q^{-1}$.
The set $\left\{ E^{i}F^{j}K^{l} | i,j\in\mathbb{N},l\in\mathbb{Z}\right\}$
can be seen to be a basis.
The parameter $q\neq 1$ in these relations makes the algebra non-commutative.
The comultiplication, counit, and antipode are defined on the basis elements as
\begin{alignat*}{3}
    \Delta(E) & =1\otimes E+E\otimes K & \Delta(F)      & =K^{-1}\otimes F+F\otimes1 & \epsilon(E) & =\epsilon(F)=0      \\
    \Delta(K) & =K\otimes K            & \Delta(K^{-1}) & =K^{-1}\otimes K^{-1}      & \epsilon(K) & =\epsilon(K^{-1})=1 \\
    S(E)      & =-EK^{-1}              & S(F)           & =-KF                       & S(K)=K^{-1} & , S(K^{-1})=K
\end{alignat*}
It can be checked that these satisfy all of the axioms of a Hopf algebra.
Importantly, the comultiplication is non-cocommutative.
If $q\to 1$, then one can show $U_{1}$ is isomorphic to the quotient $U(\mathfrak{sl}(2))[K]/(K^2-1)$ of the universal enveloping algebra.

Notice that the set generated by $E$ and $K$ is independent of $F$ or $K^{-1}$.
There is a sub-Hopf algebra generated by $E$ and $K$ with the relations above,
with basis $\left\{ E^{i}K^{l}| i\in\mathbb{N},l\in\mathbb{N}\right\}$.
We denote this sub-Hopf algebra $H\subset U_q$.
Notice the product and coproduct respect the power of $E$, so $H$ graded by the power of $E$.
Each grading $H_i=\left\langle E^{i}K^{l}|l\in\mathbb{N}\right\rangle$ is infinite dimensional.

For any basis element $E^i K^l$ we can use the commutation relations to give an explicit formula for the comultiplication:
\begin{equation}\label{eq:EKcom}
    \Delta(E^{i}K^{l})
    =\sum_{r=0}^{i}q^{r(i-r)}
    \left[\begin{smallmatrix}i\\r\end{smallmatrix}\right]
    E^{i-r}K^{l}\otimes E^{r}K^{l+(i-r)}
\end{equation}
The $q$-binomial coefficients $\left[\begin{smallmatrix}i\\r\end{smallmatrix}\right]$ are defined when $q$ is not a root of unity in terms of
the $q$-numbers $[n]=\frac{q^{n}-q^{-n}}{q-q^{-1}}=q^{n-1}+q^{n-3}+\cdots+q^{-n+3}+q^{-n+1}$.

\subsection{Main Results}
With a formula for the comultiplication in equation \ref{eq:EKcom} and the commutation relations, we can find a formula for the Hopf square
\begin{equation}
    \Psi^2(E^{i}K^{l})
    =\sum_{r=0}^{i}q^{r(i-r+2l)}
    \left[\begin{smallmatrix}i\\r\end{smallmatrix}\right]
    E^{i}K^{2l+i-r}
\end{equation}
For a fixed grading of the Hopf algebra \(H_i\), the basis is $\{E^i,E^iK,E^iK^2,\ldots\}$.
Each element of the space is determined by the power of $K\in \{0,1,\ldots\}$. This will be the state space of a Markov chain \(X_n\) with $E^i K^0$ the initial state, $X_0=0$.

The $i=0$ grading has a trivial Markov chain because \(\Psi^2(K^{l})=K^{2l}\),
so the transition probability is $T(l,j)=\delta({2l,j})$:
the state $l$ moves to the state $2l$ with probability $1$.
This deterministic doubling of our state occurs in each grading, which we will often remove from our analysis.

In the $i=1$ grading, we have two possible transitions,
\(
\Psi^2\left(EK^{l}\right)
=EK^{2l+1}+q^{2l}EK^{2l}
\).
We  break up the Hopf square map into a composition:
\(
EK^{l}\overset{D}{\mapsto} EK^{2l} \overset{\Phi_1}{\mapsto} EK^{2l+1} + q^{2l} EK^{2l}
\)
To focus in on the randomness in the process, we examine the process
\begin{equation}
    \Phi_1 (EK^l) = EK^{l+1} + q^l EK^l
\end{equation}

Since Theorem \ref{thm:Diac} of \cite{DiaconisPangRam} does not apply, the normalization factor $N_{EK^l}$ will be different for each basis element.
However, for each \(l\), $N_{EK^l}=1+q^l$, giving transition probabilities of equation \ref{eq:probs}.
To ensure these are non-trivial probabilities, we fix $q> 0$, and we note this is a Markov process as the probabilities only depend on the current state.

Let $\{X_n\}_{n=0}^\infty$ be a Markov chain with initial state \(0\) and transition matrix defined as above and equation \ref{eq:probs}.
\begin{theorem}\label{thm:Dist}
    For the Markov chain $X_n$ and $k\in \{0,1,2,\ldots\}$,
    \begin{equation}\label{eq:Dist}
        \mathbb P(X_n = k) = \frac{1}{\left(-1;q\right)_{k}}\left[\sum_{(y_{0}+\cdots+y_{k}=n-k)}\prod_{i=0}^{k}\left(\frac{q^{i}}{q^{i}+1}\right)^{y_{i}}\right]
    \end{equation}
    Where $(a;q)_n:=(1-a)(1-aq)\cdots(1-aq^{n-1})$ is the $q$-Pochhammer symbol.
\end{theorem}

The result does not give a closed form to study the distribution.
However, we know the expected number of steps to reach state $N$, the hitting time of this chain.

\begin{theorem}\label{thm:Hit}
    For the Markov chain $X_n$ and $N\in \{0,1,2,\ldots\}$, the hitting time of state $N$,
    \begin{equation}\label{eq:Hit}
        \mathbb E [\min \{n\geq 0 \mid X_n = N\}] = N+\frac{q^{N}-1}{q-1}
    \end{equation}
\end{theorem}
When $q\gg 1$, \(T(l,l)\to 1\), so we expect slow growth and a large hitting time, and when $q\to0$, we expect the opposite.
The Markov chain clearly depends heavily on the parameter $q$.
\begin{enumerate}
    \item When $q=1$, $T(l,l)=T(l,l+1)=\frac{1}{2}$, and the distribution is binomial:
          \(
          \mathbb P(X_n = k) =\binom{n}{k} \frac{1}{2^{n}}
          \)
          which means that the expectation $\mathbb{E} (X_n) = \frac{n}{2}$.

    \item When $q=0$, there is a fair coin flip at position $0$, but once \(X_n\geq 1\), the probability of success is always $1$.
          Thus, $\mathbb{E} (X_n)\approx n$.

    \item As $q\to \infty$, the fair coin flip remains at position $0$, but afterwards, $T(l,l+1)\to 0$.
          As a heuristic, taking $T(l,l+1)=0$ gives a finite chain on two states, and therefore $\mathbb{E} (X_n) = 1- (\frac{1}{2})^n$.
\end{enumerate}
This suggests that a phase transition occurs at $q=1$: when $q>1$, the expectation is sub-linear, while when $0\leq q<1$, the expectation is linear.

The above argument only uses characteristics of the probabilities.
This motivates a more general class of Markov chains, of which ours is a special case.

Let $X_{n}^{\alpha}$ be a Markov chain $X_{0}^{\alpha},X_{1}^{\alpha}\ldots$ on $\{0,1,2,\ldots\}$, $X_0^{\alpha}=0$. The transition matrix is given by ``failure" $T(i,i)=\alpha(i)$ and ``success" $T(i,i+1)=1-\alpha(i)$.
The equivalent of equation \ref{eq:Dist} of Theorem \ref{thm:Dist}, is
\begin{equation*}
    \mathbb P(X_n^\alpha = k) = \left[ \prod_{i=0}^{k-1}\left(1-\alpha(i)\right)^{y_{i}}\right]\left[\sum_{(y_{0}+\cdots+y_{k}=n-k)}\prod_{i=0}^{k}\left(\alpha(i)\right)^{y_{i}}\right]
\end{equation*}
And the hitting time, equation \ref{eq:Hit} of Theorem \ref{thm:Hit}, becomes \(\sum_{i=0}^{N-1} \frac{1}{1-\alpha(i)}\)

We can characterize the phase transition for these general models, with some moderate assumptions on the $\alpha(i)$, formalizing the previous intuition.
\begin{theorem}\label{thm:phase}
    For the Markov chain $X_n^{\alpha}$,
    if for $0\leq q <1$, the $0<\alpha(i)<1$ are strictly decreasing to $0$ as $i\to \infty$, for $q>1$, strictly increasing to $1$, and when $q=1$, identically $\frac{1}{2}$, then the Markov chain exhibits a phase transition in expectation:
    \begin{equation}
        \lim_{n\rightarrow\infty}\frac{\mathbb{E}(X^{\alpha}_{n})}{n}=
        \begin{cases}
            1           & 0<q<1 \\
            \frac{1}{2} & q=1   \\
            0           & q>1
        \end{cases}
    \end{equation}
\end{theorem}
\begin{corollary} \textbf{(Theorem \ref{thm:introTheorem})}
    The Markov chain $X_n$ undergoes a phase transition at $q=1$.
\end{corollary}
\begin{proof}
    Here, $\alpha(i)=\frac{q^{i}}{1+q^{i}}$.
    The derivative with respect to $i$ is $\frac{q^i\ln(q)}{(q^i+1)^2}$.
    Thus, $\alpha(i)$ is strictly decreasing for $q<1$ and increasing when $q>1$.
    Also see that as $i\to \infty$, when $q<1$, $\alpha(i)\to 0$ and when $q>1$, $\alpha(i)\to 1$
\end{proof}

To get better bounds on the chains $X_n^{\alpha}$ and $X_n$, we use martingale theory (see, for example, \cite{durrett}).
A function $f$ depending on $X_n$ and $n$ must satisfy the martingale condition
\begin{equation*}
    f(x,n)\alpha(x) + f(x+1,n)(1-\alpha(x)) = f(x,n-1)
\end{equation*}
Simplifying, we end up with a difference relation
\begin{equation*}
    \Delta_n(f)|_{n-1}  = -(1-\alpha(x))\Delta_x(f)|_{x}
\end{equation*}
This allows us to separate variables and guess a solution $f(x,n)=h(x) + g(n)$, where $g(n)-g(n-1)=-\lambda$ and $h(x+1)-h(x)=\lambda \frac{1}{1-\alpha(x)}$. Then
\[
    h(x)=\lambda \sum_{i=0}^{x-1} \frac{1}{1-\alpha(i)} =\lambda x  + \lambda \sum_{i=0}^{x-1} \frac{\alpha(i)}{1-\alpha(i)}
\]
Taking the two sums and setting an initial condition $\lambda = 1$, the above work gives motivation for the following theorem.
\begin{theorem}\label{thm:Mart}
    For the Markov chain $X_n^{\alpha}$, there is a martingale $\{Y_n^{\alpha}\}_{n\geq0}$ defined by
    \begin{equation}\label{eq:Y_n^alpha}
        Y_n^{\alpha} := X_n^{\alpha} - n + \sum_{i=0}^{X_n^{\alpha}-1} \frac{\alpha(i)}{(1-\alpha(i))}
    \end{equation}
\end{theorem}

\begin{corollary}
    For the Markov chain $X_n$, there is a martingale $\{Y_n\}_{n\geq0}$ defined by
    \begin{equation}\label{eq:Y_n}
        Y_n := X_n - n + \frac{1-q^{X_n}}{1-q}
    \end{equation}
\end{corollary}

\begin{proof}
    When $\alpha(i)=\frac{q^i}{1+q^i}$, the sum in equation \ref{eq:Y_n^alpha} becomes
    \(
    \sum_{i=0}^{X_n-1} {q^i} = \frac{1-q^{X_n}}{1-q}
    \)
\end{proof}

We can use the martingale of equation \ref{eq:Y_n} to characterize directly the nature of the phase transition of the Markov chain as the parameter $q$ from the quantum group is varied.
The second requires more powerful theorems but gives a stronger result.

\begin{theorem}\label{thm:q>1}
    For the Markov chain $X_n$, when $q>1$, there exists an $N$ such that for $n\geq N$,
    \[X_n\leq C \ln (n)\]
    almost surely, for some constant $C$ independent of $n$.
\end{theorem}

\begin{theorem}\label{thm:q<1}
    For the Markov chain $X_n$, when $0<q<1$, there exists a random variable $Y$ such that
    \[
        \lim_{n\to\infty} X_n - n = Y
    \]
    almost surely.

    This means that there exists an $N$ such that for $n\geq N$,
    \( X_n\geq n - C\)
    almost surely, for some constant $C$ independent of $n$.
\end{theorem}

We can also say something about the generalized chain \(X_n^\alpha\)

\begin{theorem}\label{thm:gen}
    For the Markov chain $X_n^\alpha$, if $\alpha(i)>0$ for every $i$, define
    \[
        \Tilde{h}(x)=\sum_{i=0}^{x-1} \left(\frac{\alpha(i)}{1-\alpha(i)}\right)
    \]
    Then this function is invertible on the natural numbers, and for any $\delta>0$, there exists an $N$ such that almost surely, for $n\geq N$
    \[X_n^\alpha \leq \Tilde{h}^{-1}(n^{2+\delta})\]

    Additionally, when $\sum_{i=0}^{n-1} \mathbb E \left(\frac{\alpha(X_{i}^\alpha)}{1-\alpha(X_{i}^\alpha)}\right)$ is bounded in $n$ and $\Tilde{h}(X_n^{\alpha})$ converges as $n\to\infty$, then there exists a random variable $Y^\alpha$ such that almost surely,
    \[
        \lim_{n\to\infty} X_n^\alpha - n = Y
    \]
\end{theorem}
The upper bound of the above theorem is not necessarily an improvement upon the trivial bound of $X_n^\alpha \leq n$, unless the inverse of \(\tilde{h}\) is sub-linear.
In \(X_n\)'s case, \(\tilde{h}\) was exponential and thus its inverse was logarithmic.

\section{Future Work}\label{sec:future}

Many other Markov chains can be of a similar form to $X_n^\alpha$ where the techniques we use might be helpful.
Growth chains can take multiple, different step lengths.
For example, if the transition $i\to i$ has probability $p(i)$, $i\to i+1$ probability $q(i)$, and $i\to i+2$ is $1-p(i)-q(i)$.
Then a function $f(x,n)$ would be a martingale if it satisfies the difference equation.
\begin{align*}
    \Delta_t f |_{x,n-1} & = -(1-p(x))\Delta_x f |_{x,n} - (1-p(x) -q(x))\Delta_x f |_{x+1,n}
\end{align*}
With conditions on $p,q$ where this can be solved for $f$, we could perform similar analysis as done on our chain.
Further extensions, with steps of different intervals and larger number of steps would also be interesting.

For example, the higher gradings of the quantum group have larger jumps: the second grading, $E^2K^l$, has the Hopf square
\begin{equation*}
    \Psi^2(E^{2}K^{l})=
    E^{2}K^{2l+2}
    +
    q^{(1+2l)}
    (q+q^{-1})
    E^{2}K^{2l+1}
    +
    q^{2(4l)}
    E^{2}K^{2l}
\end{equation*}
Isolating the randomness would leave us with the transition matrix
\begin{equation*}
    T\left(l, j\right)=
    \begin{cases}
        \frac{q^{2l}}{1+q^{1+l}(q+q^{-1})+q^{2l}}              & j=l   \\
        \frac{q^{(1+l)}(q+q^{-1})}{1+q^{1+l}(q+q^{-1})+q^{2l}} & j=l+1 \\
        \frac{1}{1+q^{1+l}(q+q^{-1})+q^{2l}}                   & j=l+2
    \end{cases}
\end{equation*}
We would like to study this specific example and the general case discussed above.

We would like to investigate if there are other quantum groups that, despite not being commutative nor cocommutative, have a Hopf square that gives positive coefficients that lead to a Markov chain.

And finally, we would like to investigate if there are properties of the Hopf algebra that could assist in analyzing the Hopf square, as there were in \cite{DiaconisPangRam}.

\section{Proofs from Section \ref{sec:NewWork}}\label{sec:Proofs}

\subsection*{Proof of Theorem \ref{thm:Dist}}
\begin{proof}
    Since each state has only two options, a ``failure'' (\(i\to i\)) and a ``success'' (\(i\to i+1\)), once the chain leaves a given state it is never repeated.
    Therefore we are looking for the probability of $k$ successes in $n$ trials.

    With exactly $k$ successes, one at each state $0,\ldots, k-1$, they must always contribute
    \begin{equation*}
        \left(\frac{1}{q^{0}+1}\right)\left(\frac{1}{q^{1}+1}\right)\cdots\left(\frac{1}{q^{(k-1)}+1}\right) = ((-1;q)_k)^{-1}
    \end{equation*}
    where $(a;q)_n$ is the $q$-Pochhammer symbol.

    The $n-k$ failures can occur at any of the $k+1$ locations $0,\ldots, k$.
    This is a partition of $n-k$ into $k+1$ parts, $y_0+y_1+\cdots + y_k = n-k$, where $y_i$ indicates the number of failures at state $i$.
    The probability of a specific partition of the failures is
    \begin{equation*}
        \left(\frac{q^{0}}{q^{0}+1}\right)^{y_{0}}\left(\frac{q^{1}}{q^{1}+1}\right)^{y_{1}}\cdots\left(\frac{q^{k}}{q^{k}+1}\right)^{y_{k}}
    \end{equation*}

    Thus, the likelihood of having exactly $k$ successes in $n$ trials is the sum of the probabilities of each individual path taken.
\end{proof}

\subsection*{Proof of Theorem \ref{thm:Hit}}
\begin{proof}
    If we make state $N$ an absorbing state by making a new chain with all of the transition probabilities the same except $P(N,N)=1$, then the hitting time is the same as the expected number of steps until absorption on this finite state space.
    The expected number of steps to being absorbed is equivalent to the expected number of steps to reach state $N$, as each time step can move at most one state.

    To calculate the expected time until absorption, we use the $(N+1)\times (N+1)$ matrix
    \begin{equation*}
        P=
        \left[
            \begin{array}{cccccc}
                \frac{1}{2} & \frac{1}{2}   & 0                     &                   &                           & 0                   \\
                0           & \frac{q}{1+q} & \frac{1}{1+q}         & 0                                                                   \\
                            & 0             & \frac{q^{2}}{1+q^{2}} & \frac{1}{1+q^{2}} & 0                                               \\
                \vdots      &               & 0                     & \ddots            & \ddots                    & 0                   \\
                            &               &                       & 0                 & \frac{q^{N-1}}{1+q^{N-1}} & \frac{1}{1+q^{N-1}} \\
                0           &               & \cdots                &                   & 0                         & 1
            \end{array}
            \right]
        =\left[
            \begin{array}{cc}
                Q & R     \\
                0 & I_{1}
            \end{array}
            \right]
    \end{equation*}
    where $Q$ is the $N\times N$ matrix in the upper left.
    We calculate $W=\left(I_{N}-Q\right)^{-1}$ to be
    \begin{equation*}
        W=\left[\begin{array}{ccccc}
                2      & 1+q & \dots  & 1+q^{N-2} & 1+q^{N-1} \\
                0      & 1+q & \dots  & 1+q^{N-2} & 1+q^{N-1} \\
                       & 0   & \ddots & \vdots    & \vdots    \\
                \vdots &     & \ddots & 1+q^{N-2} & 1+q^{N-1} \\
                0      &     & \cdots & 0         & 1+q^{N-1}
            \end{array}\right]
    \end{equation*}
    The theory of absorbing chains says the expected number of steps before being absorbed in any absorbing state (in our case, state $N$), starting in state $i$, is the sum of the $i^{th}$ row of $W$.
    Equivalently, it is the corresponding row of the column vector
    \begin{equation*}
        \vec{t}:=W\vec{1}=
        \left[\begin{array}{ccccc}
                \sum_{i=0}^{N-1}\left(1+q^{i}\right),   &
                \sum_{i=1}^{N-1}\left(1+q^{i}\right),   &
                \cdots,                                 &
                \sum_{i=N-2}^{N-1}\left(1+q^{i}\right), &
                1+q^{N-1}
            \end{array}\right]^T
    \end{equation*}
    Starting in state $0$, the hitting time is
    \(\sum_{i=0}^{N-1}\left(1+q^{i}\right) = N+\frac{1-q^{N}}{1-q}\).
\end{proof}

\subsection*{Proof of Theorem \ref{thm:phase}}
\begin{proof}
    For notation, collect the failure probabilities $\alpha = \{\alpha(0), \alpha(1), \ldots\}$.
    % Let $X_0^{\alpha}=0$.
    The random variables $X_{n}^{\alpha}$ follow the transition matrix
    \begin{equation*}
        T_\alpha=
        \left[\begin{array}{ccccc}
                \alpha(0) & 1-\alpha(0) & 0           & 0           & \cdots \\
                0         & \alpha(1)   & 1-\alpha(1) & 0           & \cdots \\
                0         & 0           & \alpha(2)   & 1-\alpha(2) & \ddots \\
                \vdots    & \vdots      & \ddots      & \ddots      & \ddots \\
            \end{array}\right]
    \end{equation*}
    Where the first row and column correspond to state $0$.
    See that once $X_{n}^{\alpha}=1$, the failure probabilities no longer use $\alpha(0)$, they only depend on $\Sigma \alpha = \{\alpha(1),\alpha(2),\ldots\}$.

    For example, if $X_{1}^{\alpha}=1$ (we succeed at $t=0$), $X_{n}^{\alpha}$ from then on follows the above transition matrix with the first row and column removed, labeled $T_{\Sigma \alpha}$
    The new first row and column correspond to state $1$.
    The new random variable $X_{n}^{\Sigma\alpha}$ defined by \(T_{\Sigma\alpha}\) is of the same form, with $X_{0}^{\Sigma\alpha}=1$.
    On the other hand, if $X_{1}^{\alpha}=0$, then \(X_{2}^{\alpha}\) comes from $T_\alpha$ and the chain has lost one time step.

    Conditioning the expectation on the result of the first time step, we get
    \begin{align*}
        \mathbb{E}(X^{\alpha}_{n}) & = P(X^{\alpha}_{1}=1)\mathbb{E}(X^{\alpha}_{n}|X_{1}^{\alpha}=1)+P(X^{\alpha}_{1}=0)\mathbb{E}(X^{\alpha}_{n}|X_{1}^{\alpha}=0) \\
                                   & =(1-\alpha(0))\mathbb{E}(X_{n-1}^{\Sigma\alpha})+\alpha(0)\mathbb{E}(X_{n-1}^{\alpha})
    \end{align*}
    And we are only looking for the asymptotic behavior of the expectation in the large $n$ limit, when $n/(n-1)\approx 1$
    \begin{equation*}
        \lim_{n\rightarrow\infty}\frac{\mathbb{E}(X^{\alpha}_{n})}{n}
        =(1-\alpha(0))\lim_{n\rightarrow\infty}\frac{\mathbb{E}(X_{n-1}^{\Sigma\alpha})}{n-1}
        +\alpha(0)\lim_{n\rightarrow\infty}\frac{\mathbb{E}(X_{n-1}^{\alpha})}{n-1}
    \end{equation*}
    By re-indexing, since $\alpha(0)\neq1$
    \begin{equation*}
        \lim_{n\rightarrow\infty}\frac{\mathbb{E}(X^{\alpha}_{n})}{n}=\lim_{n\rightarrow\infty}\frac{\mathbb{E}(X_{n}^{\Sigma\alpha})}{n}
    \end{equation*}

    The same process works when starting with $X^{\Sigma\alpha}_{n}$, or $X^{\Sigma^2\alpha}_{n}$: conditioning on the first time steps of these random variables,
    \begin{equation*}
        \lim_{n\rightarrow\infty}\frac{\mathbb{E}(X^{\Sigma^k\alpha}_{n})}{n}=\lim_{n\rightarrow\infty}\frac{\mathbb{E}(X_{n}^{\Sigma^{k+1}\alpha})}{n}
    \end{equation*}
    Chaining the equalities, for every $k$, $\lim_{n\rightarrow\infty}\frac{\mathbb{E}(X^{\alpha}_{n})}{n}=\lim_{n\rightarrow\infty}\frac{\mathbb{E}(X_{n}^{\Sigma^{k}\alpha})}{n}$

    In the $q<1$ case, the $(1-\alpha(n))$ are strictly increasing, and thus bounded below by the first, $1-\alpha(0)$.
    This implies that the success probabilities for $X_{n}^{\Sigma^k\alpha}$ are bounded below by $1-\alpha(k)$.
    Since our random variable is the number of successes in $n$ trials, then the number of trials times the lowest possible probability of success is a lower bound on the expectation:
    \begin{equation*} \label{eq:ExpBound}
        \mathbb{E}(X^{\Sigma^k\alpha}_{n})\geq n(1-\alpha(k))
    \end{equation*}
    Thus, when $q<1$, we get the bound
    \(
    \lim_{n\rightarrow\infty}\frac{\mathbb{E}(X^{\alpha}_{n})}{n}\geq1-\alpha(k)
    \)
    for every $k$.
    Since $\lim_{k\to\infty}(1-\alpha(k))=1$,
    \begin{equation}
        \lim_{n\rightarrow\infty}\frac{\mathbb{E}(X^{\alpha}_{n})}{n}=1
    \end{equation}

    The other case works similarly: if $q>1$, then the probability of success is strictly decreasing, so the probability of success at the first state is now an upper bound for all of the success probabilities.
    This means that for every $k$,
    \begin{equation*}
        \lim_{n\rightarrow\infty}\frac{\mathbb{E}(X^{\alpha}_{n})}{n}\leq 1-\alpha(k)
    \end{equation*}
    for any $k\geq 1$.
    Then $\lim_{k\to\infty}(1-\alpha(k))=0$ implies that $\lim_{n\rightarrow\infty}\frac{\mathbb{E}(X^{\alpha}_{n})}{n}=0$

    When $q=1$, the probability of success at each stage is always $\frac{1}{2}$, so the expectation of the binomial distribution divided by $n$ is $\frac{1}{2}$, and the limit holds through.
\end{proof}

\subsection*{Proof of Theorem \ref{thm:Mart}}
\begin{proof}
    Equation \ref{eq:Y_n^alpha} is claimed to be a Martingale:
    \begin{align*}
        \mathbb E [Y_n^{\alpha} \mid Y_{n-1}^{\alpha}]
         & = \mathbb E \left[X_n^{\alpha} - n + \sum_{i=0}^{X_n^{\alpha}-1} \frac{\alpha(i)}{(1-\alpha(i))} \bigg\vert X_{n-1}^{\alpha}\right]                                                          \\
         & = X_{n-1}^{\alpha}\alpha(X_{n-1}^{\alpha}) + (X_{n-1}^{\alpha}+1)(1-\alpha(X_{n-1}^{\alpha})) - n                                                                                            \\
         & \qquad + \alpha(X_{n-1}^{\alpha})\sum_{i=0}^{X_{n-1}^{\alpha}-1} \frac{\alpha(i)}{(1-\alpha(i))} + (1-\alpha(X_{n-1}^{\alpha}))\sum_{i=0}^{X_{n-1}^{\alpha}} \frac{\alpha(i)}{(1-\alpha(i))} \\
         & = X_{n-1}^{\alpha}  - (n-1)  + \sum_{i=0}^{X_{n-1}^{\alpha}-1} \frac{\alpha(i)}{(1-\alpha(i))}  = Y_{n-1}^{\alpha}
    \end{align*}
\end{proof}

\subsection*{Proof of Theorem \ref{thm:q>1}}
\begin{proof}
    First, since $Y_0= X_0 - 0 +\frac{1-q^{0}}{1-q}=0$ and for every $n$,
    \[
        \mathbb E [Y_n] = \mathbb E [ \mathbb E [Y_n|Y_{n-1}]] = \mathbb E [ Y_{n-1}]
    \]
    then $\mathbb E [Y_n] = \mathbb E [ Y_{0}]=0$ for every $n$.
    The definition of \(Y_n\) implies that $\mathbb E [X_n] =  n - \mathbb E [\frac{q^{X_n}-1}{q-1}]$, and as $X_n$ is non-negative, its expectation must be non-negative. Therefore, regardless of $q$,
    \begin{equation*}
        \mathbb E \left[\frac{q^{X_n}-1}{q-1}\right] \leq n
    \end{equation*}
    Using Markov's Inequality, we then know that for any $\lambda>0$,
    \begin{equation*}
        \mathbb P \left( \frac{q^{X_n} - 1}{q-1} \geq \lambda\right) \leq
        \frac{1}{\lambda} \mathbb E \left[\frac{q^{X_n}-1}{q-1}\right] \leq
        \frac{n}{\lambda}
    \end{equation*}
    Define the events $E_n := \left\{\frac{q^{X_n} - 1}{q-1} \geq \lambda\right\}$ for each $n\geq 0$.
    To use the Borel-Cantelli lemma, we want the sum of the probability of the events to be finite.
    So we choose $\lambda=n^{2+\delta}$ for some fixed $\delta>0$, so we have the bound $P(E_n)\leq 1/n^{1+\delta}$, and thus
    \begin{equation*}
        \sum_{n=0}^\infty \mathbb P ( E_n) \leq \sum_{n=0}^\infty \frac{1}{n^{1+\delta}} <\infty
    \end{equation*}
    By Borel-Cantelli, with probability $1$, only finitely many of the events $E_n$ can occur.
    Thus, there is an $N$ such that $E_n$ does not occur for $n\geq N$.
    For those $n$, since $q>1$,
    \begin{equation*}
        \frac{q^{X_n} - 1}{q-1} \leq n^{2+\delta}\implies {X_n} \leq \log_{q}\left(n^{2+\delta}({q-1}) +1\right) \leq C \ln(n)
    \end{equation*}
    which gives a logarithmic bound of $X_n$ almost surely when $q>1$.
\end{proof}

\subsection*{Proof of Theorem \ref{thm:q<1}}
\begin{proof}
    % We first use Azuma's inequality, 

    %Second way
    We first show that $\lim_{n\to\infty}X_n = \infty$ almost surely.

    We do so by using Azuma's inequality (see \cite{Azuma}), that states that if $\{Y_n\}$ is a martingale and $|Y_n-Y_{n-1}|\leq c_n$ for constants $c_n$, then for any $\epsilon>0$
    \begin{equation*}
        \mathbb P (Y_n-Y_0 \leq -\epsilon) \leq \exp\left(\frac{-\epsilon^2}{2\sum_{k=1}^n c^2_k} \right)
    \end{equation*}
    Note that since $q<1$
    \begin{align*}
        |Y_n -Y_{n-1}| & = | X_n + \frac{q^{X_n}-1}{q-1} -n - X_{n-1} - \frac{q^{X_{n-1}}-1}{q-1} +n-1| \\
                       & \leq \frac{q^{X_{n-1}+1}-q^{X_{n-1}}}{q-1} + 1 = 1+ q^{X_{n-1}} \leq 2
    \end{align*}
    Thus, since $Y_0=0$ we know that
    \begin{equation*}
        \mathbb P (X_n + \frac{q^{X_n}-1}{q-1} -n \leq -\epsilon) \leq \exp\left(\frac{-\epsilon^2}{8n} \right)
    \end{equation*}
    We again want to use Borel-Cantelli, so fixing some $\delta>0$, let $\epsilon =n^{1/2 + \delta}$.
    Then the sum of the probability of the events $E_n = \{X_n + \frac{q^{X_n}-1}{q-1} -n \geq n^{1/2 + \delta}\}$ is finite.
    For some \(N\), if $n\geq N$, we have with probability $1$
    \begin{equation}\label{eq:Azuma}
        X_n + \frac{q^{X_n}-1}{q-1} -n  \geq -n^{1/2 + \delta}
        \implies
        X_n \geq n - n^{1/2 + \delta} - \frac{1}{1-q}
    \end{equation}
    The last inequality following from the bound $\frac{q^x-1}{q-1}\geq \frac{1}{1-q}$ when $q<1$.
    Thus, $X_n\to\infty$ as $n\to\infty$ almost surely.

    Secondly, we show that the Martingale has bounded variance, which shows it is bounded in $L^1$, and thus we can use the Martingale convergence theorem.
    See that, as in the previous proof, $\mathbb E Y_n = \mathbb E Y_0 = Y_0 = 0 $, so
    \begin{align*}
        \mathbb E [Y^2_n-(\mathbb E Y_n)^2] & = \mathbb E [Y_n^2- (\mathbb E Y_0)^2] = \mathbb E [Y_n^2-Y_{0}^2] = \sum_{i=1}^n \mathbb E [ Y^2_i-Y^2_{i-1}] \\
        % &= \sum_{i=1}^n \mathbb E [ \mathbb E[Y^2_i -Y^2_{i-1} | Y_0,\ldots, Y_{i-1}]] \\
                                            & = \sum_{i=1}^n \mathbb E [ \mathbb E[Y^2_i-2Y_iY_{i-1}+Y^2_{i-1}| Y_0,\ldots, Y_{i-1}]]
        % &= \sum_{i=1}^n \mathbb E [ \mathbb E[(Y_i-Y_{i-1})^2| Y_0,\ldots, Y_{i-1}]]  \\
        = \sum_{i=1}^n \mathbb E [(Y_i-Y_{i-1})^2]
    \end{align*}
    With probability $ \frac{q^{X_{n-1}}}{1+q^{X_{n-1}}}$, $X_n=X_{n-1}$, so
    \begin{align*}
        Y_n -Y_{n-1}
        % &=  X_n + \frac{q^{X_n}-1}{q-1} -n - X_{n-1} - \frac{q^{X_{n-1}}-1}{q-1} +n-1 \\
         & = X_{n-1} + \frac{q^{X_{n-1}}-1}{q-1} -n - X_{n-1} - \frac{q^{X_{n-1}}-1}{q-1} +n-1
        % \\
        % &
        =  -1
    \end{align*}
    Otherwise, with probability $ \frac{1}{1+q^{X_{n-1}}}$,
    \begin{align*}
        Y_n -Y_{n-1}
        % &=  X_n + \frac{q^{X_n}-1}{q-1} -n - X_{n-1} - \frac{q^{X_{n-1}}-1}{q-1} +n-1 \\
         & = X_{n-1} + 1 + \frac{q^{X_{n-1} + 1}-1}{q-1} -n - X_{n-1} - \frac{q^{X_{n-1}}-1}{q-1} +n-1 \\
         & =   \frac{q^{X_{n-1} + 1}- q^{X_{n-1}}}{q-1}
        % \\
        % &=   \frac{q^{X_{n-1}}(q- 1)}{q-1} \\
        % &
        =   q^{X_{n-1}}
    \end{align*}
    Thus, any increment satisfies
    \begin{align*}
        \mathbb E [(Y_i- Y_{i-1})^2] %&= \mathbb E [\mathbb E [(Y_i-Y_{i-1})^2|Y_{i-1}]] \\
         & = \mathbb E [ (-1)^2 \frac{q^{X_{i-1}}}{1+q^{X_{i-1}}} +  q^{2X_{i-1}} \frac{1}{1+q^{X_{i-1}}}]
        % &= \mathbb E [ \frac{q^{X_{i-1}}}{1+q^{X_{i-1}}} (1+ q^{X_{i-1}})] \\
        = \mathbb E [ q^{X_{i-1}}]
    \end{align*}
    Which means that the variance of the martingale is $\sum_{i=1}^n \mathbb E [ q^{X_{i-1}}]$.
    We want an upper bound on $q^{X_{i-1}}$ when $0<q<1$ to show that this sum is uniformly bounded.

    Equation \ref{eq:Azuma} says that $q^{X_{i-1}}\leq Cq^{i-\sqrt{i}}$, so the variance of the martingale is at most $C\sum_{i=1}^n [ q^{i-\sqrt i}]$, which converges as \(n\to\infty\).
    Thus, the martingale is bounded in $L^2$ and therefore $\sup_{n} [Y_n] <\infty $.

    The Martingale Convergence theorem applies (see \cite{durrett}).
    As a result, there exists some random variable $Y$ such that $\lim_{n\to\infty}Y_n = Y_\infty$ almost surely.
    Combining this with the fact that $q^{X_n}\to 0$, then we have the first result, that if $Y=Y_\infty -\frac{1}{1-q}$,
    \(
    \lim_{n\to\infty} X_n - n = Y
    \)
    For some $N$ close enough to the convergence, any $n\geq N$ gives $X_n - n + \frac{1-q^n}{1-q} - Y_\infty > -\epsilon$.
    Taking a possibly larger $N$ for the other limit gives, almost surely
    \begin{equation*}
        X_n -n  >  Y_\infty - \frac{1}{1-q}- \epsilon = - C
    \end{equation*}
\end{proof}

\subsection*{Proof of \ref{thm:gen}}
\begin{proof}
    Equation \ref{eq:Y_n^alpha} gives a Martingale in terms of $X_n^\alpha$,
    \begin{equation*}
        Y_n^{\alpha} = X_n^{\alpha} - n + \sum_{i=0}^{X_n^{\alpha}-1} \frac{\alpha(i)}{(1-\alpha(i))}
    \end{equation*}
    The function $\Tilde{h}(x)$, in the case of the previous theorems, is $\frac{q^{X_n}-1}{q-1}$.
    The argument from the proof of Theorem \ref{thm:q>1} follows in the same way: $\Tilde{h}(X^\alpha_n)$ replaces $\frac{q^{X_n}-1}{q-1}$ everywhere, giving the result that for any $\delta>0$ there exists some $N$, if $n\geq N$
    \begin{equation*}
        \Tilde{h}(X_n^\alpha) \leq n^{2+\delta}
    \end{equation*}
    almost surely.
    Note that $\Tilde{h}(x)< \Tilde{h}(x+1)$ as each $\frac{1}{1-\alpha(x)}$ is strictly positive, so $\Tilde{h}$ has an inverse.
    Thus we get the bound
    \begin{equation*}
        X_n^\alpha \leq \Tilde{h}^{-1}(n^{2+\delta})
    \end{equation*}

    For an upper bound, the proof of Theorem \ref{thm:q<1} requires the variance be bounded:
    \[
        \mathbb E [Y_n^2 - (\mathbb E Y_n)^2] = \sum_{i=1}^n \mathbb E \left(\frac{\alpha(X_{i-1}^\alpha)}{1-\alpha(X_{i-1}^\alpha)}\right)
    \]
    When the variance above is uniformly bounded in $n$, we can use the Martingale Convergence Theorem to say that $X_n^\alpha - n + \Tilde{h}(X^\alpha_n)\to Y_\infty$.
    Even further, when $\Tilde{h}(X^\alpha_n)\to C$, we get that $X_n^\alpha - n$ approaches a random variable.

\end{proof}

\backmatter

\bibliography{HopfAlgBib.bib}% common bib file

\end{document}